\pdfoutput=1
\nonstopmode
\documentclass{amsart}
\usepackage{epsfig}
\usepackage[all]{xy}
\usepackage{amssymb,amscd}
\usepackage{enumerate}
\usepackage[numbers]{natbib}
\usepackage{todonotes,bbm,MnSymbol,wasysym}
\usepackage[all]{xy}
\usepackage{mathtools}
\mathtoolsset{showonlyrefs}
\newdir{ >}{{}*!/-10pt/@{>}}
\newdir^{ (}{{}*!/-5pt/@^{(}}
\newdir_{ (}{{}*!/-5pt/@_{(}}
\def\cgaps#1{}
\def\Cgaps#1{}
\allowdisplaybreaks

\def\undersetbrace#1\to#2{\underbrace{#2}_{#1}}        
\def\oversetbrace#1\to#2{\overbrace{#2}^{#1}}
\def\AMSunderset#1\to#2{\underset{#1}{#2}}
\def\AMSoverset#1\to#2{\overset{#1}{#2}}

\newtheorem{thm}{Theorem}

\newtheorem*{prop*}{Proposition}

\newtheorem{conj}[thm]{Conjecture}
\newtheorem*{thm*}{Theorem}
\newtheorem{lem}[thm]{Lemma}
\newtheorem*{lem*}{Lemma}
\newtheorem{cor}[thm]{Corollary}
\newtheorem*{rem*}{Remark}
\newtheorem*{cor*}{Corollary}

\newcommand{\nmb}[2]{\ifx!#1{\ref{nmb:#2}}%
\else\if.#1{\label{nmb:#2}}%
\else\if0#1{\label{nmb:#2}}%
\else{{#2}}%
\fi\fi\fi}
 
\newcommand{\eean}{\end{eqnarray*}}
\newcommand{\benu}{\begin{enumerate}}
\newcommand{\eenu}{\end{enumerate}}
\newcommand{\bea}{\begin{eqnarray}}
\newcommand{\eea}{\end{eqnarray}}

\def\im{{\rm im}}

\newcommand{\Z}{{\mathbb Z}}
\newcommand{\R}{{\mathbb R}}

\def\o{\circ\,}
\def\X{\mathfrak X}

\def\si{\sigma}

\def\ph{\varphi}

\def\Ga{\Gamma}
\def\De{\Delta}

\def\i{^{-1}}

\def\ver{\on{\ver}}
\def\Emb{\on{Emb}}
\def\Imm{\on{Imm}}

\def\Diff{\on{Diff}}
\def\Vol{\operatorname{Vol}}
\def\x{\times}
\def\p{\partial} 
\def\X{{\mathfrak X}}
  
\def\L{\mathcal{L}}

\def\R{{\mathbb R}}

\def\exp{\operatorname{exp}}

\let\on=\operatorname

\def\Tr{\on{Tr}}
\def\vol{\on{vol}}

\def\x{\times}
\def\p{\partial}
\let\on=\operatorname

\def\L{\mathcal L}
\def\AMSonly#1{}

\def\p{\partial}

\def\R{\mathbb R}
\def\<{\big\langle}
\def\>{\big\rangle \:}
\def\div{\operatorname{div}}
\def\grad{\operatorname{grad}}
\def\Nor{\operatorname{Nor}}

\begin{document}
\title{Riemannian geometry of the space of volume preserving immersions}
\author[M. Bauer]{Martin Bauer}
\address{Faculty for Mathematics, Technical University of Vienna,  Austria}
\email{bauer.martin@univie.ac.at}

\author[P. Michor]{Peter W. Michor}
\address{Department of Mathematics, University of Vienna, Austria}
\email{peter.michor@univie.ac.at}

\author[O. M\"uller]{Olaf M\"uller}
\address{Department of Mathematics, University of Regensburg, Germany}
\email{Olaf.Mueller@mathematik.uni-regensburg.de }

\thanks{All authors were partially supported by the Erwin Schr\"{o}dinger Institute programme: 
Infinite-Dimensional Riemannian Geometry with Applications to Image Matching and Shape Analysis. 
M. Bauer was supported by the FWF-project P24625 (Geometry of Shape spaces).}%
\subjclass{Primary 58B20, 58D15}
 \keywords{Volume Preserving Immersions; Sobolev  Metrics; Well-posedness; Geodesic Equation}

\date{\today}%

\begin{abstract}
Given a compact manifold $M$ and a Riemannian manifold $N$ of bounded geometry, we consider the manifold $\on{Imm} (M,N)$ of immersions from $M$ to $N$ and its subset $\on{Imm}_\mu (M,N)$ of those immersions with the property that the volume-form of the pull-back metric equals $\mu$. We first show that the non-minimal elements of $\on{Imm}_\mu (M,N) $ form a splitting submanifold. On this submanifold we consider the Levi-Civita connection for various natural Sobolev metrics write down the geodesic equation and show local well-posedness in many cases. The question is a natural generalization of the corresponding well-posedness question for the group of volume-preserving diffeomorphisms, which is of great importance in fluid mechanics. 
\end{abstract}

\maketitle
\setcounter{tocdepth}{1}
\tableofcontents

\section{Introduction.}
Let $M$ be a compact connected (oriented) $d$-dimensional manifold, and let $(N,\bar g)$ be a Riemannian manifold of bounded geometry.
In this article we study Riemannian metrics on the space $\Imm_{\mu}(M,N)$ of all immersions from $M$ to $N$ that preserve a fixed volume form 
$\mu$; i.e., those immersions $f$ such that $\on{vol}(f^*\bar g)=\mu$.

The interest in this space can be motivated from applications in the study of biological membranes, where the volume density of the  surface remains constant 
during certain biological deformations. Another source of interest can be found in connections to the field of mathematical hydrodynamics, as the space $\Imm_{\mu}(M,N)$
can be seen as a direct generalization of the group of all volume preserving diffeomorphisms. As a consequence the geodesic equations studied in Sect.~\ref{sec:L2} can be seen as an analogue of 
Euler's equation for the motion of an incompressible fluid. We will employ similar methods as Ebin and Marsden \cite{EM1970} to study the wellposedness of some of the equations that appear in the context 
of (higher order) metrics on $\Imm_{\mu}(M,N)$. Finally, the analysis of this article can be seen as a direct continuiation of the analysis of Preston \cite{Preston2011,Preston2012,PS2013} for the motion and geometry of the space and whips and chains, which would 
correspond to the choice $M=S^1$ or $M=[0,1]$ and $N=\mathbb R^2$. In Sect.~\ref{sec:ex} we will compare the results of this article, with some of the results obtained in these already  better investigated situations.

We will consider the space $\Imm_{\mu}(M,N)$ as a subspace of the bigger space of all smooth immersions from $M$ to $N$. Another interesting space that appears in this context is the space $\on{Imm}_{g}(M,N)$ of all isometric immersions; i.e., 
all immersions that pull back $\bar g$ to a fixed metric $g$ on $M$. Similarily, one can consider all these spaces in the context of embeddings as well.
We have the following diagram of inclusions:
\begin{align*}
&\on{Imm}_{g}(M,N)\subset\on{Imm}_{\mu}(M,N)\subset \on{Imm}(M,N)\,,\\
&\on{Emb}_{g}(M,N)\subset\on{Emb}_{\mu}(M,N)\subset \on{Emb}(M,N)\,.
\end{align*}
Here $\on{Emb}_{g}(M,N)$ and $\on{Emb}_{\mu}(M,N)$ are defined similar as for the bigger spaces of immersions. 
We will concentrate in this article on the space $\on{Imm}_{\mu}(M,N)$ (resp. $\on{Emb}_{\mu}(M,N)$) and we plan to consider the geometry of the space of isometric immersions (embeddings) in future work. 

In the article \cite{Molitor2012} it has been shown that the space $\on{Emb}^{\times}_{\mu}(M,N)$ is a smooth tame splitting submanifold of the space of all smooth embeddings $\on{Emb}(M,N)$,  
where the elements of the spaces $\on{Emb}^{\times}_{\mu}(M,N)$ are assumed to have nowhere vanishing second fundamental form. The choice of this space is not very fortunate 
for our purposes for various reasons; e.g., in the case of closed surfaces in $\mathbb R^3$ this condition restricts to convex surfaces only. 
Thus, as a first step,  we want to ged rid of that additional condition and show a similar statement for the spaces in the above diagram. 
Similar, as in \cite{Molitor2012}, the proof of this statements will be an application of the Nash-Moser inverse function theorem, however we will have to consider a different splitting of the tangent space.
The proof of these statements will be given in Sect. \ref{Imm_vol}. However we will still be forced to require the immersions to be not minimal; i.e., 
they do not have an everywhere vanishing  mean curvature. 
In the case of embeddings into $\mathbb R^3$, the absence of compact minimal embeddings already shows that this is only a weak restriction. For the space of all volume preserving embeddings the submanifold 
result has been shown in \cite{GV2014}, using a different method of proof.

In the second part of this article we will equip the space $\on{Imm}(M,N)$ with the family of reparametrization invariant Sobolev metrics as introduced in \cite{Bauer2011b,Bauer2012d}:
\begin{align*}
G_f(h,k)=\int_M \bar g((1+\Delta)^lh,k)\mu,\qquad l\in \mathbb N \,.
\end{align*}
Here $\Delta$ denotes the Bochner-Laplacian of the pullback metric $g=f^*\bar g$. See also \cite{Bauer2014} for an overview on various metrics on spaces of immersions.
In this article we will be interested in the induced metric of these metrics on the submanifold $\on{Imm}_{\mu}(M,N)$. In particular we  
will discuss the orthogonal projection from 
$T\Imm(M,N)$ to $T\Imm_{\mu}(M,N)$ with respect to these metrics, consider the induced geodesic equation on the submanifold, 
and give sufficient conditions on the order $l$ to ensure local well-posedness of the corresponding geodesic equations.

We will conclude the article with the two special cases of volume preserving diffeomorphisms ($M=N$) and constant speed parametrized curves ($M=S^1$, $N=\mathbb R^2$).

\section{The manifold of immersions}\label{sec:notation}
Let $M$ be a compact connected (oriented) finite dimensional manifold, and let $(N,\bar g)$ be a Riemannian manifold of bounded geometry.
To shorten notation we will somtimes write $d$ to  denote the dimension of the manifold $M$; i.e., $\dim(M)=d$. 
Let $\Emb(M,N)$ be the space of all 
smooth embeddings $M\to N$. It is a smooth manifold modelled on Fr\'echet spaces. 
The tangent space at $f$ of $\Emb(M,N)$ equals $\Gamma(f^*TN)$, the space of sections of $TN$ along 
$f$, and the tangent bundle equals the open subset of $C^\infty(M,TN)$ consisting of those 
$h:M\to TN$ such that $\pi_N\o h \in \Emb(M,N)$. 
More generally, let $\Imm(M,N)$ be the smooth Fr\'echet manifold of all smooth immersions $M\to N$. Similarly to
$\Emb(M,N)$ the  tangent bundle equals the open subset of $C^\infty(M,TN)$ consisting of those 
$h:M\to TN$ such that $\pi_N\o h \in \Imm(M,N)$. See \cite{KrieglMichor97} as a general reference for calculus in infinite 
dimensions, and for nearly all spaces that will be used here. From here onwards we will only work on the more general space of immersions, 
however all results continue to hold for embeddings as well.

Following the presentation in \cite{Bauer2011b} we also introduce the Sobolev completions of
the relevant spaces of mappings. 
In the canonical charts for $\Imm(M,N)$ centered at an immersion $f_0$, 
every immersion corresponds to a section of
the vector bundle $f_0^*TN$ over $M$ (see \cite[section~42]{MichorG}). 
The smooth Hilbert manifold $\Imm^k(M,N)$ (for $k>\dim(M)/2+1$) 
is then constructed by gluing together the Sobolev completions
$H^k(f_0^*TN)$ of each canonical chart.
One has
$$\Imm^{k+1}(M,N) \subset \Imm^k(M,N),\qquad \bigcap_{k}\Imm^k(M,N) = \Imm(M,N)\,.$$
Similarly, Sobolev completions of the space $T\Imm(M,N) \subset C^\infty(M,TN)$ are defined
as $H^k$-mappings from $M$ into $TN$; i.e., $T\Imm^k(M,N)=H^k(M,TN)$.  
More information can be found in \cite{Shubin1987} and in \cite{EichhornFricke1998}.

In the following we will introduce some notation, that we will use throughout the article.
For $f\in \Imm(M,N)$ we denote by $g=f^*\bar g$ the pullback metric on $M$; we use $g$ if 
we need short notation, and $f^*\bar g$ if we stress the dependence on $f$.

The \emph{normal bundle} $\Nor(f)$ of an immersion $f$ is a sub-bundle of $f^*TN$ 
whose fibers consist of all vectors that are $\bar g$-orthogonal to the image of $f$:
$$\Nor(f)_x = \big\{ Y \in T_{f(x)}N : \forall X \in T_xM : \bar g(Y,Tf.X)=0  \big\}.$$
If $\dim(M)=\dim(N)$ then the normal bundle is the zero vector bundle. 
Any vector field $h$ along $f \in \Imm(M,N)$ can be decomposed uniquely 
into parts {\it tangential} and {\it normal} to $f$ as
\begin{equation}
\label{decomposition}
h=Tf.h^\top + h^\bot,
\end{equation}
where $h^\top$ is a vector field on $M$ and $h^\bot$ is a section of the normal bundle $\Nor(f)$. 

Let $X$ and $Y$ be vector fields on $M$. 
Then the covariant derivative $\nabla^g_X Tf.Y$ splits into tangential and a normal parts as
$$\nabla_X Tf.Y=Tf.(\nabla_X Tf.Y)^\top + (\nabla_X Tf.Y)^\bot = Tf.\nabla_X Y + S(X,Y).$$
$S=S^f$ is the \emph{second fundamental form of $f$}. 
It is a symmetric bilinear form with values in the normal bundle of $f$. 
When $Tf$ is seen as a section of $T^*M \otimes f^*TN$ one has $S=\nabla Tf$ since
$$S(X,Y) = \nabla_X Tf.Y - Tf.\nabla_X Y = (\nabla Tf)(X,Y).$$
The trace of $S$ is the \emph{vector valued mean curvature} $\Tr^g(S) \in \Ga\big(\Nor(f)\big)$.

\subsection{Riemannian metrics on spaces of immersions}
A Riemannian metric $G$ on $\Imm(M,N)$ is a section of the bundle
$$L^2_{\on{sym}}(T\Imm(M,N);\R)$$ 
such that at every $f \in \Imm(M,N)$, $G_f$ is a symmetric positive definite bilinear mapping 
$$G_f: T_f\Imm(M,N) \x T_f\Imm(M,N) \to \R.$$
Each metric is {\it weak} in the sense that $G_f$, seen as a mapping
$$G_f: T_f\Imm(M,N) \to T^*_f\Imm(M,N)\,,$$
is injective (but it can never be surjective).

\begin{rem*}
We require that our
metrics will be invariant under the action of $\on{Diff}(M)$, hence the
quotient map dividing by this action will be a Riemannian submersion off the orbifold singularities of the quotient; 
see  \cite{CerveraMascaroMichor91}.
This means that the tangent map of the quotient map $\Imm(M,N)\to \Imm(M,N)/\Diff(M)$ is 
a metric quotient mapping between all tangent spaces.
Thus we will get Riemannian metrics on the quotient space $\Imm(M,N)/\Diff(M)$.
\end{rem*}

All of the metrics we will look at will be of the form 
\begin{equation}
G^L_f(h,k) = \int_{M} \bar g( L^f h, k)\, \vol(f^*\bar g)
\end{equation}
where $L^f:T_f\Imm(M,N) \to T_f \Imm(M,N)$ is a positive bijective operator depending smoothly on $f$,
which is selfadjoint unbounded in the Hilbert space completion of $T_f\Imm(M,N)$ with inner product $G^{L^2}_f$.
Here $G^{L^2}$ denotes the metric that is induced by the operator $L=\on{Id}$; i.e.,
\begin{equation}
G^{L^2}_f(h,k) = \int_{M} \bar g(h, k)\, \vol(f^*\bar g).
\end{equation}

We will assume in addition that $L$ is equivariant with respect to reparametrizations; i.e.,
$$L^{f\circ\ph}=\ph^* \o L^f \o (\ph\i)^* = \ph^*(L^f)\qquad \text{for all }\ph\in\on{Diff}(M).$$
Then the metric $G^L$ is invariant under the action of $\on{Diff}(M)$ as required above.

In this article we will focus on integer order Sobolev metrics; i.e., metrics of the form:
\begin{equation}
 G^l_f(h,k) = \int_{M} \bar g( (1+\Delta_g)^l h, k)\, \vol(f^*\bar g)
\end{equation}
where $\Delta$ is the Bochner Laplacian of the pullback metric $g=f^*\bar g$ and where $l\in \mathbb N$.

\begin{thm}\label{thm:wellposedness}
Let $G^L$ be the metric induced by the operator $L=(1+\Delta_g)^l$.
\begin{enumerate}
\item \label{WeakAndStrong} For any $l\geq 0$ and $k>\operatorname{min}\left(\frac{\operatorname{dim}(M)}{2}+1,l\right)$, the metric $G^L$ extends to a \emph{smooth weak Riemannian metric} on the Hilbert manifold $\Imm^{k}(M,N)$. 
For $l=k\in 2\mathbb N$ the metric extends to a \emph{strong Riemannian metric} on the Hilbert manifolds $\Imm^k(M,N)$.
 \item \label{SuperStrongExistence} For any $l\geq 1$ and $k>\frac{\operatorname{dim}(M)}{2}+1$ the initial value problem for the geodesic equation 
has unique local solutions both in the Hilbert manifold $\Imm^{k+2l}$ and in the Fr\'echet manifold $\Imm(M,N)$.
The solutions depend smoothly on $t$ and on the initial
conditions $f(0,\;.\;)$ and $f_t(0,\;.\;)$. Moreover the Riemannian exponential mapping $\exp$ exists 
and is smooth on a neighborhood of the zero section in the tangent bundle, 
and $(\pi,\exp)$ is a diffeomorphism from a (smaller) neighbourhood of the zero 
section to a neighborhood of the diagonal.
\end{enumerate}
\end{thm}

\begin{rem*}
This theorem also holds for general (non-integer) Sobolev order $l\in \mathbb R_{\ge 0}$, but it needs some technical tools which will be developed in a future paper.
\end{rem*}

\begin{proof}
To prove the first statement of Item \eqref{WeakAndStrong} we rewrite the metric as
\begin{equation*}
G_f^L(h,k)=\int_M \bar g\left( (1+\Delta_g)^{\lfloor l/2\rfloor}(h), (1+\Delta_g)^{\lceil l/2 \rceil}(k)\right)\, \operatorname{vol}(f^*\bar g)\;.
\end{equation*}
We need to show that this extends to a smooth Riemannian metric on the Sobolev completion of high enough order, which is not trivial as the operator 
$\Delta_g$ has non-smooth coefficients (the coefficients depend on the foot point immersion $f$). To prove this statement we write $\Delta_g$ in local coordinates:
\begin{equation}
\Delta_g(h)=\frac{1}{\sqrt{|g|}}\partial_i \left(\sqrt{|g|}g^{ij}\partial_j h      \right)\;.
\end{equation}
Note, that $g=f^{*}\bar g$ is of  regularity $H^{k-1}$ as the immersion $f$ is of  regularity $H^k$. Using carefully the Sobolev embedding theorem one can thus show that $\Delta$ 
smoothly extends to an operator field
\begin{align*}
\Imm^k(M,N)&\to L(H^s(M, TN),H^{s-2}(M, TN))\\
f&\mapsto \Delta_g
\end{align*}
for $2\leq s\leq k$, see also \cite{Mueller2015}, which proves the first assertion.



For the second assertion of Item \eqref{WeakAndStrong}, we pick an immersion $f_0$ and a 
standard convex neighborhood $U$ of $f_0$; i.e., we choose a covering of $f_0(M)$ via $N$-convex neighborhoods $W_i$, define $V_i:= f_0^{-1} (W_i)$ 
and require from each element $f$ of $U$ that $f(V_i) \subset W_i$. 
We have to compare the translationally invariant metric of the modelling Hilbert space $G_{f_0}$ with the pointwise metric $G_f$ and have to show that there is 
$$A \in C^{\infty} \left(U, BL (T_{f_0}\Imm^k(M,N), T_{f_0}\Imm^k(M,N))\right) $$ 
with $G_f (h,k) = G_{f_0} (A(f) h, k) $ such that $A(f)$ is an isomorphism for all $f \in U$. 
Here tangential vectors in the chart neighborhood are identified using the vector space structure. As we define charts via $\operatorname{exp}_N$ this can be expressed via $d \operatorname{exp}_N$, 
that is, via Jacobi fields in the target space $N$. We denote the corresponding map by $J$, which is an isomorphism $J: \Gamma_{H^s}( f_0^* T N) \rightarrow \Gamma_{H^s} (f^* T N) $ for all $s \in \Z$ with $s \leq k,$ 
because it is precomposition with a $H^k $ map as the chosen neighborhoods are convex.  
We can choose $A(f) :=   J \o L_{f_0}^{-1} \o J^{-1} \o L_f $, 
which is in fact an isomorphism, as $L_f:H^l \rightarrow H^{-l}$ is an isomorphism for all $f \in \Imm^l(M,N)$, 
which in turn is a direct consequence of $k$-safeness as laid down in \cite{Mueller2015}.

A proof of the Item $(2)$ is contained in \cite{Bauer2011b}, modulo the fact that one needs $k$-safeness of the corresponding operator as explained in detail in \cite{Mueller2015}. 
\end{proof}

In Sects. \ref{sec:L2} and \ref{sec:Hk} we will consider the restriction of these metrics to the submanifold of all
volume form preserving immersions.

\section{The submanifold structure of the space of volume preserving immersions}
In this section we will study the manifold structure of the space of all volume preserving immersions. 
For technical reasons which will become clear in the proof of the theorem, we restrict ourselves to non minimal immersions, which we will denote by
\begin{equation}
\on{Imm}^*_{\mu} (M,N):=\left\{ f\in\on{Imm}(M,N): \operatorname{Tr}^g(S)\neq 0 \right\}\;.
\end{equation}
Note, that this is not a pointwise condition, but that we are only excluding those immersions 
whose second fundamental form vanishes identically. 

In the following we will show that $\on{Imm}_{\mu}^*(M,N)$ is a submanifold of the manifold of immersions. We will follow the proof of \cite{Molitor2012}, but with much less restrictive conditions.
\begin{thm}\label{Imm_vol}
 The space $\on{Imm}_{\mu}^* (M,N)$ is a tame splitting Fr\'echet submanifold of $\on{Imm} (M,N)$, and the tangent space at an element $f \in \on{Imm}_{\mu} (M,N)$  is naturally isomorphic (via the postcomposition with the exponential map) to 
\begin{equation}\label{tangenspace_Immvol}
 T_{f,\mu}\Imm(M,N):= \{h\in T_f\Imm(M,N): \div^\mu (h^\top) - \bar g(h^\bot,\Tr^g(S)) =0 \}. 
\end{equation}
The same is true for Sobolev completions $\on{Imm}_{\mu}^{*,k} (M,N)$ of order $k> \frac{d}{2} +1$.
\end{thm}
\begin{rem*}
This result is stronger then the one in \cite{Molitor2012}. There it has only been shown that the space
of embeddings with nowhere vanishing mean curvature is a a tame splitting Fr\'echet submanifold of $\on{Emb} (M,N)$. 
The condition of having a nowhere vanishing mean curvature is, however, too restrictive for our purposes, since it only allows for convex (resp. concave)
surfaces in the case of hypersurfaces.
\end{rem*}

\begin{rem*}
Our main subject will be spaces of {\em immersions}, but note that the proofs of most theorems, e.g., of the previous one, immediately carry over to the smaller spaces of {\em embeddings}. 
\end{rem*}

The first step in the proof of the theorem (modeled after the proof of the corresponding theorem in \cite{Molitor2012}) is the following proposition which allows to decompose any vector field $h$ along $f$ into a  part
$h_\mu$ that is divergence-free -- in the sense that its flow preserves the volume $\mu$ -- and its complement.

\begin{lem}\label{lem:decomp:immvol}
Let $f\in \Imm(M,N)$ with $\vol(f^*\bar g)=\mu$.
Then for each tangent vector  $h\in T_f\Imm(M,N)$ there exist   
$$
h_\mu\in T_{f,\mu}\Imm(M,N)
\quad\text{  and }\quad
p\in \begin{cases} 
C^\infty(M) \quad &\text{if}\quad\Tr^g(S)\ne 0 \\
                       C^\infty(M)/\mathbb R \quad &\text{if}\quad\Tr^g(S)= 0 
				 \end{cases}
$$
such that 
\begin{equation}\label{p_l2_proj}
h= h_\mu + Tf.\on{grad}^g(p) + p.\Tr^g(S).
\end{equation}
The field $h_\mu$ is uniqely determined by \eqref{p_l2_proj}. But $p$ is uniquely 
determined only if $\Tr^g(S)\ne 0$ and is unique up to an additive constant if $\Tr^g(S)=0$. 
However, the decomposition \eqref{p_l2_proj} is unique in both cases, and depends smoothly on 
$h\in T\Imm(M,N)$ and for $f\in \Imm(M,N)$ with $\Tr^g(S)\neq 0$ also on $f$. The mappings
\begin{align}
&P^1:T_f\Imm(M,N)\to T_f\Imm(M,N), &\quad& P^1(h) = h_\mu,
\\
&P^2:T_f\Imm(M,N)\to \X(M), && P^2(h) = \grad^g(p),
\\
&P^3:T_f\Imm(M,N)\to \Gamma (M,f^* TN), && P^3(h) = p.\Tr^g(S).
\end{align}
are parts of smooth fiber linear  homomorphisms of vector 
bundles over $\Imm(M,N)$. The same is true for $h$ and $f$ from some Sobolev class $H^k$; in this case, $p$ is of Sobolev class $H^{k+1}$.
\end{lem}
It should be mentioned that in the case of $M=N$, this decomposition is exactly the Helmholtz-Hodge decomposition, 
as the tangential vectors in $T_{f,\mu}\Imm(M,N)$ are co-closed and so can be decomposed as sums of harmonic forms 
and coexact forms. This result, for $f$ an embedding of an 
oriented compact manifold with nowhere vanishing mean curvature $\Tr^g(S)$, is due to Molitor \cite[proposition 1.4]{Molitor2012}.
Since we claim here more, we shall sketch a (slightly different) proof.

\begin{proof} If we can write $h$ as \eqref{p_l2_proj}, we can apply $\div^g$ to the tangential part of 
\eqref{p_l2_proj} and apply $\bar g(\quad,\Tr^g(S))$ to the normal part of \eqref{p_l2_proj} to obtain 
\begin{align}
\div^g(h^\top) &= \div^g\big((h_{\mu})^\top\big) + \De^g(p) 
= \bar g\big((h_{\mu})^\bot,\Tr^g(S)\big)+ \De^g(p)
\quad{\text{and}}
\\
\bar g(h^\bot,\Tr^g(S)) &= \bar g\big((h_{\mu})^\bot,\Tr^g(S)\big) + p \|\Tr^g(S)\|^2_{\bar g},
\text{which combine to} 
\\
(\De^g &- \|\Tr^g(S)\|^2_{\bar g})(p) = \div^g(h^\top) -  \bar g(h^\bot,\Tr^g(S)).
\end{align}
Now, as a selfadjoint elliptic differential operator, 
for any $k\in \mathbb N \cup \{ \infty \}$ and Sobolev space $H^{k}(M)$  of functions, 
$\De^g-\|\Tr^g(S)\|^2_{\bar g}:H^{k,g}(M)\to H^{k-2,g}(M)$ has index 
zero. Moreover, by Hopf's maximum principle (see \cite[page~96]{Aubin98}, carried over to a compact 
manifold) the kernel of 
$D:= \De^g-\|\Tr^g(S)\|^2_{\bar g}$ on $C^2$ functions is contained in the space of constant functions.

If $\Tr^g(S)\ne 0$ then this kernel is zero and 
$\De^g-\|\Tr^g(S)\|^2_{\bar g}:H^{k}(M)\to H^{k-2}(M)$ is a linear isomorphism, in particular $D^{-1}$ is a well-defined bounded linear map $H^k \rightarrow H^{k+2}$ for all $k > n/2$, and we get a unique
function 
\begin{equation}
p = L^{-1} (  div^g(h^\top) -  \bar g(h^\bot,\Tr^g(S))) \in L^{-1} (H^{k-1}) = H^{k+1}.
\end{equation}
Then the desired decomposition is
\begin{align}
h_{\mu}&= h - Tf.\grad^g(p) - p.\Tr^g(S),\quad \text{  because}
\\
\div^g\big(h_{\mu}^\top\big) &= \div^g(h^\top) - \De^g(p) = 
\bar g\bigl(h_{\mu}^\bot,\Tr^g(S)\bigr).
\end{align}

If $\Tr^g(S)=0$ then let $h^\top = h^\top_{\mu} + \grad^g(p)$ be the Helmholtz-Hodge 
decomposition of $\X(M)$, where now $p$ is only unique up to an additive constant.
Put $h_{\mu}= Tf.h^\top_{\mu} + h^\bot$ and get the desired decomposition.

Let us give a second argument.
If $\Tr^g(S)=0$ then $\ker(\De^g-\|\Tr^g(S)\|^2_{\bar g})=\ker(\De^g)=\mathbb R$.
Recall also the Hodge decomposition 
$C^\infty(M)= \De^g(C^\infty(M)) \oplus H^0_{\text{dR}}(M)=\De^g(C^\infty(M)) \oplus \mathbb R$.
Thus $\De^g:C^\infty(M)/\mathbb R\to C^\infty(M)/\mathbb R$ is a linear isomorphism, and the above 
proves the tangential Helmholtz-Hodge decomposition, and we use 
$h_{\mu}=Tf.h^\top_{\mu} + h^\bot$.
\end{proof}

For the proof of our main statement --  Theorem \ref{Imm_vol} -- we will need two further lemmas. 
Therefore we will need the following definitions:
\begin{itemize}
 \item With $\rho(f) \in C^\infty(M)$ we denote the Radon-Nikodym derivative of the volume density $\vol(f^* g)$ with respect to the background density $\mu$; i.e.,  
 \begin{equation*}
  \vol(f^* g) = \rho(f)\mu.
 \end{equation*}
 \item Let $P_f := \rho (f) \o \phi_f^{-1}$, with $\phi_f$ is a standard exponential chart around $f$ and 
$Q_f : h \mapsto (h , P_f (h)-1 )$.  
\end{itemize}
Now Lemma 1.7 from \cite{Molitor2012} holds for our new definition of $h_\mu$:
\begin{lem}
The map $P_f$ is a smooth tame map. Its derivative $$dP^1_f : (h,k) \mapsto d_h P_f \cdot k$$ is, for fixed $h$, a linear partial differential operator of degree $1$, and its coefficients are 
partial differential operators of degree $1$ in the variable $h$. Moreover, for all $k \in \Gamma(f^* TN)$, we have

$$  d_0P^1_f \cdot k = \big(  \on{div}^g (k^T \o f^{-1}) - g(k^\perp \o f^{-1}, \Tr^g(S))  \big) \o f \cdot P^1_f (0).$$ 
\end{lem}
The proof of this lemma is the same as the proof of \cite[Lemma 1.7]{Molitor2012}.

The second lemma that we will need corresponds to Lemma 1.9 in \cite{Molitor2012}:

\begin{lem}
For any $f \in \Imm^*_{\mu} (M,N)$, the smooth tame map 
$$Q_f : \phi_f (U_f) \rightarrow T_f\Imm(M,N) \oplus C^\infty (M, \R)$$ 
is invertible on an open neighborhood $U$ of the zero section. Its inverse on $U$ is also a smooth tame map. The corresponding statement for finite Sobolev order holds as well.
\end{lem}

\begin{proof}
Writing down the equations one sees directly that $d_X Q_f$ is invertible if and only if $A_f: p \mapsto d_X P_f ({\rm grad} (p) + H_f \cdot p)$ is invertible. Moreover, as in the first lemma, $d_X P_f$ is, for $X$ sufficiently small, elliptic and of index $0$, thus $A$ is, for $X$ sufficiently small, elliptic of order $2$ and of index $0$. Now we can use the strong maximum principle to show injectivity and thus surjectivity: Assume $Pp=0$. Then, as $M$ is connected, $p(M) = [a,b]$ for some $a,b \in \R$. Assume $a \neq b$. Sard's theorem implies that we can find $r \in (a,b)$ such that $p^{-1} (r)$ is a smooth hypersurface of $M$. Now define the codimension-$0$ submanifold-with-boundary $D := p^{-1}( (r, b])$ and choose $d \in \partial D$ and a chart 
 $ u: V \rightarrow W \subset \R^n$ around $d$ with $V \cap U = \emptyset$. Then we apply the strong maximum principle in the following form: Let $L$ be a strictly elliptic operator of second order on functions in an open connected domain $W$ of $\R^n$ with zero order term $L_0 \leq 0$. If $p \in C^2 (W) \cap C^0 (\overline{W})$ with $Lp \geq 0$ in $W$, then $p = \sup \{ p(W) \} $ or $p$ does not attain a nonnegative maximum in $W$. All the assumptions are satisfied by $L=A$, as $M$ is 
compact and therefore $P$ is strictly elliptic, and as $L_0 = A_0 = - \| \Tr^g(S) \|^2 \leq 0$. Therefore $p(W)= \{ p(y) \} $, in contradiction to the fact that $d \in \partial D$. Thus $a=b$, thus $p$ is constant, and as $\vert \on{Tr}(g\i S) \vert^2 $ is positive and does not vanish identically, this implies $p=0$, which, together with the consideration of the index, concludes the proof.
\end{proof}

\begin{proof}[Proof of Theorem \ref{Imm_vol}]
Using the previous lemmas the theorem follows directly by an application of the tame Fr\'echet inverse function theorem. \end{proof}

\section{The $L^2$-geometry}\label{sec:L2}

From here on we will only treat the space $\Imm_{\mu}(M,N)$ and we will equip it with the restriction of the invariant 
$L^2$-metric on the space of all immersions
\begin{equation}
G^{L^2}_f(h,k) = \int_{M} \bar g(h, k)\, \vol(f^*\bar g).
\end{equation}
Since we keep the volume density on $\Imm_{\mu}(M,N)$ constant and since the invariant $L^2$-metric depends only on the 
volume density, the restriction of the non-invariant $L^2$-metric 
\begin{equation}
\bar G_f (h,k) := \int_{M} \bar g(h, k)\, \mu
\end{equation}
to $T\Imm(M,N)|_{\Imm_\mu(M,N)}$
equals the restriction of the invariant metric.
The exponential mapping for $\bar G$ is simply $\big(\exp^{\bar G}_f(h)\big)(x) = \exp^{\bar g}_{f(x)}(h(x))$ and similarly for curvature; see \cite{Kainz84}.
As a first step we want now to consider the orthogonal projection from $T\Imm(M,N)$ to $T\Imm_{\mu}(M,N)$ with respect to the invariant $L^2$ metric which equals the orthogonal projection with respect to $\bar G$.

\begin{thm}\label{thm:projection}
Let $P$ be the mapping  
\begin{align}
\label{proj}
&P_f: T_f\Imm(M,N)\to T_f\Imm(M,N)\\
&P_{f}(X)= X - Tf.\on{grad}^g(p) - p.\Tr^g(S),  
\end{align}
where $p$ is the solution to
\begin{equation}
\label{cond-p}
 \Delta p - \| \Tr^g(S) \|^2  p = \on{div}^g(X^T) - g(X^\perp , \Tr^g(S)) \,.
\end{equation}
Then $P_f$ is the orthogonal projection from $T_f\Imm(M,N)$ to $T_f\Imm_{\mu}(M,N)$
with respect to the invariant $G^{L^2}$-metric.
\end{thm}
\begin{proof}
We first show, that the mapping $P$ has values in the correct space. Therefore we check the determining equation of $T_f\Imm_{\mu}(M,N)$:
\begin{align*}
 &\operatorname{div}^g(P_f(X)^{\top})-\bar g(P_f(X)^{\bot},\operatorname{Tr}^g(S))\\&\qquad= \operatorname{div}^g(X^{\top})-\bar g(X,\operatorname{Tr}^g(S))-
  \operatorname{div}^g(\operatorname{grad}^g(p))+\bar g(p.\operatorname{Tr}^g(S),\operatorname{Tr}^g(S))\;,
\end{align*}
which vanishes by the definition of the function $p$. For the $L^2$-orthogonality, we compute 
\begin{align*}
&G^{L^2}_f\left(X_{\mu}, Tf.\grad^g (p) + p \cdot \Tr^g (S)\right)\\&\qquad=
\int_{M}\bar g\left( X_{\mu}, Tf.\grad^g (p) + p \cdot \Tr^g (S)\right) \on{vol}(f^*\bar g)\\
&\qquad=\int_{M}\bar g\left( Tf.X_{\mu}^{\top}, Tf.\grad^g (p)\right)\on{vol}(f^*\bar g)+\int_{M} p.\bar g\left( X_{\mu}^\bot,\Tr^g (S)\right) \on{vol}(f^*\bar g)\\
&\qquad = \int_{M} g\left(X_{\mu}^{\top},\grad^g (p)\right)\on{vol}(f^*\bar g)+\int_{M} p.\bar g\left( X_{\mu}^\bot,\Tr^g (S)\right) \on{vol}(f^*\bar g)\\
&\qquad = -\int_{M} \on{div}^g(X_{\mu}^{\top}).p \on{vol}(f^*\bar g)+\int_{M} p.\bar g\left( X_{\mu}^\bot,\Tr^g (S)\right) \on{vol}(f^*\bar g).\\
\end{align*}
Here the last step consists of an integration by parts.
Using the characterization for the tangent space $T_f \on{Imm}_{\mu}(M,N)$ we obtain:
\begin{equation}
G^{L^2}_f\left(X_{\mu}, Tf.\grad^g (p) + p \cdot \Tr^g (S)\right)=0\,.\qedhere
\end{equation}
\end{proof}

For the space of volume preserving diffeomorphisms ($M=N$), Ebin and Marsden have showed 
that the projection extends smoothly to Sobolev completions of high enough order. They then used this result 
to conclude the smoothness of the geodesic spray and obtained as a consequence the local well-posedness of the geodesic equation. 
However, it turns out that the smoohtness of the projection is not true anymore in our situation:
\begin{lem}
For $M\neq N$ and any $k\in \mathbb R$ 
the projection $P$ is not a continuous map on the Sobolev completions of order $k$ 
\begin{equation}
P: \Imm^{k}(M,N)\times T \Imm^{k}(M,N) \rightarrow T \Imm^{k}_{\mu}(M,N)\;.
\end{equation}
\end{lem}

\begin{rem*}
Note that for high enough $k$ the projection $P_f$ is smooth for a fixed foot point $f \in   \Imm^{k+2}_{\mu}(M,N)$; i.e., seen as a map
\begin{equation}
P_f:   T_f \Imm^{k}(M,N) \rightarrow T_f \Imm^{k}_{\mu}(M,N).
\end{equation}
This is in accordance with the the results of \cite{Preston2011,Preston2012} for the space of arclength-parametrized curves.
\end{rem*}

\begin{proof}
The non-smoothness of the projection follows immediately from the appearance of the term $\Tr^g(S)$ in the definition of the projection. 
This term contains second derivatives of the foot point $f$, which entails the last assertion. To see this 
take any $f \in H^k$ with $0  \neq \Tr^g (S) \in H^{k-2} $ but not in $H^{k}$ . Then it is easy to find $h= h^\perp $ such that there is a real number $c$ with $ c \cdot \vert \vert \Tr^g (S) \vert \vert^2 = g(h, \Tr^g (S) ))  $, e.g., by choosing $h := c \cdot \Tr^g (S)$. 
Uniqueness then implies that $p=c$, so if $P_f(h) \in H^k $, then $c \cdot  \Tr^g(S) = h - P_f(H) \in H^k  $, which yields a contradiction. 
\end{proof}

\subsection{The geodesic equation}\label{sec:L2:geodesic}

In the following we want to calculate the geodesic equation on the space of volume preserving immersions. To do this, we first 
calculate the covariant derivative of the $L^2$--metric on $\on{Imm}_{\mu}(M,N)$. Therefore, we will use the same method as in \cite{Preston2012}. We shall also use 
$\nabla^{\bar G}$, the Levi-Civita covariant derivative on $\on{Imm}$ for the non-invariant metric $\bar G$ which coincides with the covariant derivative induced by the Levi-Civita covariant derivative $\nabla^{\bar g}$ of the metric $\bar g$ on $N$; see 
\cite[3.7]{Bauer2011b} and \cite{Kainz84}. 
\begin{thm}
The covariant derivative of the $L^2$-metric on $\on{Imm}_{\mu}(M,N)$ is given by
$$\left\{\begin{aligned}
\nabla_{U} V &= \nabla^{\bar G}_U V -  Tf.\on{grad}^g(p) - p.\on{Tr}^g(S),\\
\left(\Delta  - \| \Tr^g(S) \|^2\right)  p &= \on{Tr}\left(\partial_t\left(g^{-1}\right)\langle \nabla^{g} V,Tf   \rangle\right) \\
&-\on{Tr}\left(g^{-1}\langle \partial_t\left( \nabla^{g(t)}\right) V,Tf   \rangle\right)  -\on{Tr}\left(g^{-1}\langle \nabla^{g} V,\nabla_t\left(Tf(t)\right)   \rangle\right)   \,.
\end{aligned}\right.$$
\noindent where $f(t)$ is a curve in $\on{Imm}_\mu (M,N)$ with $\partial_t f (0) = U (f)$, and where, for a bilinear form $H $ on $T_{f_0}\Imm(M,N)$ we use the short-hand notation $H \langle \cdot, \cdot \rangle =  H \otimes \overline{g} $ defining a bilinear form on $ TM^* \otimes f^* TN$, i.e. for vector fields along $f$. Note that the right-hand side of the second equation contains no $t$-derivative of $f_t$. Therefore the same is true for the first equation.
\end{thm}

\begin{proof}
Using the submanifold structure of  $\on{Imm}_{\mu}(M,N)$ the covariant derivative can be calculated as
\begin{align}
\nabla_{\partial_t} f_t =  \nabla^{\Imm}_{\partial_t} f_t  - S^{\on{Imm}_{\mu}}(f_t,f_t),
\end{align}
here $S^{\on{Imm}_{\mu}}(f_t,f_t)$ denotes the second fundamental form of $$\on{Imm}_{\mu}(M,N)\subset \left(\on{Imm}(M,N),\bar G\right)\,.$$
We follow closely the proof of \cite{Preston2012} to calculate the second fundamental form.
Let  $U$ and $V$ be vector fields  on $\on{Imm}_{\mu}(M,N)$, with value $u$ and $v$ when evaluated at $\gamma$. Then the second fundamental form 
is given by
\begin{align}
S^{\on{Imm}_{\mu}}(u,v)= ((\nabla^{\bar G}_U V)_{\gamma})^{\bot},
\end{align}
here $(\cdot)^{\bot}$ denotes the orthogonal projection onto the normal bundle with respect to both the invariant metric $G^{L^2}$ or the non-invariant metric $\bar G$ which coincide along $\Imm_{\mu}$.

Now let $f(t)$ be a curve of volume preserving immersions with $f(0)=\gamma$ and let $V(t)\in T_{f(t)}\on{Imm}_{\mu}(M,N)$ be a curve along $f(t)$. 

It remains to calculate the orthogonal projection of $V_t(0)$. 
To shorten the notation we will write $f=f(0)$. 
Using the formula for the projection of Thm.~\ref{thm:projection}, we obtain that
\begin{align}
(V_t(0))^{\bot} = Tf. \on{grad}^{g} p_{uv} + p_{uv}.\on{Tr}^g(S), 
\end{align}
where $p_{uv}$ is the solution to
\begin{align}
 (\Delta  - \| \Tr^g(S) \|^2 ) p_{uv} = \on{Tr}\left(g(t)^{-1}\langle \nabla^{g(t)} V_t,Tf(t)   \rangle\right)|_{t=0}.
\end{align}
Here we prefer the second last expression of Lemma \ref{partialtvolume} to the last one; the reason is that the term 
$\operatorname{div}^g f_t^T$ contains $t$-derivatives of $f_t$ due to the presence of a term $(Tf)^{-1}$ because of Equation \ref{decomposition}. In the case of curves, the metric $g$ was independent of the time $t$. In the higher dimensional case this is not true anymore.
Since $V(t) \in T_f\on{Imm}_{\mu}(M,N)$ we have:
\begin{align}
\on{Tr}\left(g(t)^{-1}\langle \nabla^{g(t)} V,Tf(t)   \rangle\right)=0 ,
\end{align}
for all $t$. Taking the derivative of this yields, using the product rule for $\partial_t$ and torsion-freeness of the pull-back covariant derivative,
\begin{align}
&\on{Tr}\left(g(t)^{-1}\langle \nabla^{g(t)} V_t,Tf(t)   \rangle\right)
=-\on{Tr}\left(\partial_t\left(g(t)^{-1}\right)\langle \nabla^{g(t)} V,Tf(t)   \rangle\right)  
\\&\qquad-\on{Tr}\left(g(t)^{-1}\langle \partial_t\left( \nabla^{g(t)}\right) V,Tf(t)   \rangle\right)  -\on{Tr}\left(g(t)^{-1}\langle \nabla^{g(t)} V,\nabla_t\left(Tf(t)\right)   \rangle\right)  \qedhere
\end{align}
\end{proof}

We are now able to write down the formula of the geodesic equation on the space of volume preserving immersions. 
To simplify the presentation we will only write the geodesic equation for the special case $N= \mathbb R^n$:
\begin{thm}
The geodesic equation of the $L^2$-metric on $\on{Imm}_{\mu}(M,\R^n)$ is given by
$$\left\{\begin{aligned}
f_{tt} &= Tf.\on{grad}^g(p) + p.\on{Tr}^g(S),\\
\left(\Delta  - \| \Tr^g(S) \|^2\right)  p &= \on{Tr}\left(\partial_t\left(g(t)^{-1}\right)\langle \nabla^{g} f_t,Tf   \rangle\right) \\
&\qquad-\on{Tr}\left(g^{-1}\langle \partial_t\left( \nabla^{g(t)}\right) f_t,Tf   \rangle\right)  -\on{Tr}\left(g^{-1}\langle \nabla^{g} f_t,\nabla_t\left(Tf(t)\right)   \rangle\right)   \,.
\end{aligned}\right.$$
\end{thm}

\begin{proof}
To obtain the formula for the geodesic equation we need to calculate the covariant derivative 
in the ambient space $(\on{Imm}(M,\R^n),\bar G)$ of $V$ in direction $u=f_t(0)$:
\begin{align}
\left(\nabla^{\bar G}_u v\right)_{\gamma} = V_t(0). 
\end{align}
Here we used the flatness of the space $\left(\on{Imm}(M,\R^n),\bar G\right)$ and the identification of $T_x \R^n $ with $\R^n$.
\end{proof}

In Sect.~\ref{sec:ex:curves} we will show, that this equation simplifies to the equation of \cite{Preston2011} for the special case $M=S^1$, $N=\mathbb R^2$.

\section{Higher order metrics}\label{sec:Hk}

In this part we consider the restriction of higher order Sobolev metrics
\begin{equation}
 G^L_f(h,k) = \int_{M} \bar g( (1+\Delta_g^l) h, k)\, \vol(f^*\bar g):= \int_{M} \bar g( L_f h, k)\, \mu,
\end{equation}
to the space of volume preserving immersions.
Since the volume form remains constant we equivalently write these metrics as
\begin{equation}
 G^L_f(h,k) = \int_{M} \bar g( (1+\Delta_g^l) h, k)\, \mu,
\end{equation}
For $l=0$ this equals the $L^2$-metric from Section~\ref{sec:L2}.

Similar as for the $L^2$-metric we are interested in the orthogonal projection to $T_f\Imm_{\mu}(M,N)$ also for these higher order metrics. Therefore we need to introduce the operator
$\Psi$:
\begin{align*}
&\Psi_f: C^{\infty}(M)\to C^{\infty}(M)\\
&\Psi_f(p)=\on{div}^g \left(\left(L^{-1}\left(Tf.\on{grad}^g(p)  + p.\Tr^g(S)\right)\right)^\top \right)\\
 &\qquad\qquad-\bar g\left (L^{-1}\left(Tf.\on{grad}^g(p) + p.\Tr^g(S)\right), \Tr^g(S)\right)\;.
\end{align*}

In the next lemma we collect some basic properties for the operator $\Psi^L_f$, that we will later use to prove the existence of the orthogonal projection.
\begin{lem}
Let $L$ be an elliptic positive $L^2$-self-adjoint pseudo differential operator of order $l$. 
Then the operator $\Psi^L_f$ is an elliptic and $L^2$-selfadjoint pseudo differential operator of order $2-2l$.  
\end{lem}

\begin{proof}
Let $q\in C^{\infty}(M)$. We have
 \begin{align*}
 &\int_M \Psi_f(p).q \operatorname{vol}(f^*\bar g)= \int_M \on{div}^g \left(\left(L^{-1}\left(Tf.\on{grad}^g(p)  + p.\Tr^g(S)\right)\right)^\top \right).q\\
 &\qquad\qquad-\bar g\left (L^{-1}\left(Tf.\on{grad}^g(p) + p.\Tr^g(S)\right), \Tr^g(S)\right).q \operatorname{vol}(f^*\bar g)\\&=
 \int_M -g\left( \left(\left(L^{-1}\left(Tf.\on{grad}^g(p)  + p.\Tr^g(S)\right)\right)^\top \right), \on{grad}^g(q)\right)\\
  &\qquad\qquad-\bar g\left (Tf.\on{grad}^g(p) + p.\Tr^g(S), L^{-1}(q.\Tr^g(S))\right) \operatorname{vol}(f^*\bar g)\\&=
  \int_M -\bar g\left(L^{-1}\left(Tf.\on{grad}^g(p)  + p.\Tr^g(S)\right) , Tf.\on{grad}^g(q)\right)\\
  &\qquad\qquad-\bar g\left (Tf.\on{grad}^g(p) + p.\Tr^g(S), L^{-1}(q.\Tr^g(S))\right) \operatorname{vol}(f^*\bar g)\\&=
  \int_M -\bar g\left(Tf.\on{grad}^g(p)  + p.\Tr^g(S) ,L^{-1}( Tf.\on{grad}^g(q))\right)\\
  &\qquad\qquad-\bar g\left (Tf.\on{grad}^g(p) + p.\Tr^g(S), L^{-1}(q.\Tr^g(S))\right) \operatorname{vol}(f^*\bar g)\\&=
  \int_M -g\left(\on{grad}^g(p)  ,(L^{-1}( Tf.\on{grad}^g(q)))^{\top}\right)-p.\bar g\left(\Tr^g(S) ,L^{-1}( Tf.\on{grad}^g(q))\right)\\
  &\qquad\qquad-g\left (\on{grad}^g(p), (L^{-1}(q.\Tr^g(S)))^{\top}\right)
  -p\bar g\left (\Tr^g(S), L^{-1}(q.\Tr^g(S))\right) \operatorname{vol}(f^*\bar g)
  \\&=
  \int_M p.\operatorname{div}\left((L^{-1}( Tf.\on{grad}^g(q)))^{\top}\right)-p.\bar g\left(\Tr^g(S) ,L^{-1}( Tf.\on{grad}^g(q))\right)\\
  &\qquad\qquad+p \operatorname{div}\left((L^{-1}(q.\Tr^g(S)))^{\top}\right)
  -p\bar g\left (\Tr^g(S), L^{-1}(q.\Tr^g(S))\right) \operatorname{vol}(f^*\bar g)\\
  &=
  \int_M p.\operatorname{div}\left((L^{-1}( Tf.\on{grad}^g(q)+q.\Tr^g(S)))^{\top}\right)\\&\qquad\qquad
  -p.\bar g\left(\Tr^g(S) ,L^{-1}( Tf.\on{grad}^g(q)+q.\Tr^g(S))\right)
  \\&= \int_M p.\Psi_f(q) \operatorname{vol}(f^*\bar g)
 \end{align*}
This proves that the operator is selfadjoint with respect to the $L^2$-metric.

We want to examine ellipticity of the pseudodifferential operator $\Psi$:
\begin{align*}
&\Psi_f: C^{\infty}(M)\to C^{\infty}(M)\\
&\Psi_f(p)=\on{div}^g \left(\left(L^{-1}\left(Tf.\on{grad}^g(p)  + p.\Tr^g(S)\right)\right)^\top \right)\\
 &\qquad\qquad-\bar g\left (L^{-1}\left(Tf.\on{grad}^g(p^l) + p.\Tr^g(S)\right), \Tr^g(S)\right)\;.
\end{align*}
Ellipticity means here that the principal symbol is nondegenerate. Let us calculate the symbol of $\Psi_f$. We use the following definition:

For a fiber-preserving and fiberwise linear pseudodifferential operator $P$ of degree $l$ between vector bundles $\pi_1 $ and $\pi_2$ over a manifold $M$, for $v \in T^*_qM$ and $x \in \pi_1^{-1} (q)$ we take 
\bea
\label{symboldef}
\sigma_P (v) (x) := \lim_{a \rightarrow \infty} \si_{P,a} (v) (x), \qquad  \si_{P,a} (v) (x) :=   a^{-l} e^{-au} P(e^{au} X) ,
\eea  
for any section $X$ of $\pi_1$ with $X(q) = x$ and any $d_q u  = v$. For the following calculation all that is needed is the property that the principal symbol is linear and multiplicative, coincides with the usual one on differential operators and is connected in the usual way to the order of the operator, which represents an algebra homomorphism from the set of pseudodifferential operators to $\R$. 

In our setting, we have  the operators 
\begin{align*}
P_1&:= {\rm grad}^g : C^\infty (M) \rightarrow V (M)\\
P_2 &:= (L^{-1} ( Tf . \cdot ) )^T = U^{-1}: V(M) \rightarrow \Gamma (f^* \tau_N ) \rightarrow V(M) \\
P_3 &:= \operatorname{div}: V(M) \rightarrow C^\infty (M)
\end{align*}
where $U:= (L \o Tf)^T: V(M) \rightarrow V(M)$. Dropping lower-order terms, we see that 
$$\si_{\Psi_f} = \si_{P_3 \o P_2 \o P_1} = \si_{P_3} \o \si_{P_2} \o \si_{P_1}$$ 
using the multiplicativity of $\si$. Now one calculates easily $ \si_{P_1} (v) (x) = \sharp_g (v) x$ for $x \in C^\infty (M)$  and $\si_{P_3} (v) (x) = v(x) $ for $x \in V(M)$. For $P_2$, we use multiplicativity once more to show that $\si_{P_2} = \si_U^{-1}$, and if $ L := (1 + \Delta)^l  $ then, of course $U \neq (\Delta^g)^l$ even if $L = \Delta^l$ but on the level of symbols we do have $\si_U = \si_{(\Delta^g)^l} =  \si_{\Delta^g}^l = g^{l}$, thus $\si_{P_2} = g^{-l}$ and all in all we get
$$\si_{\Psi_f} (v) (x) = g^{-l} (v) \cdot v(x \cdot \sharp_g (v)) = g^{1-l} (v) \cdot x, $$ 
which is indeed nondegenerate. The same holds for general $L$.
\end{proof}

This allows us to define the analogue of the orthogonal projection also for these higher order metrics:
\begin{thm}
Let $f $ be an immersion of Sobolev class $H^s$ and let $P$ be the mapping  
\begin{align}
\label{higher-projection}
&P^L_{f}(X)= X - L^{-1}(Tf.\on{grad}^g(p) + p.\Tr^g(S)),  
\end{align}
where $p$ is the unique solution of
\begin{equation}\label{generalmetric_p}
\Psi_f^L(p) =  \on{div}^g(X^T) - g(X^\perp , \Tr^g(S)) = \Tr^g \big(    \overline{g} (\nabla X, Tf )   \big)
\end{equation}
then $P^L_f(X)$ is the $G^L$-orthogonal projection onto $T_f\Imm_{\mu}(M,N)$. It is linear and smooth.
\end{thm}

\begin{proof}
Let us show that the equations above indeed well-define a smooth linear projection. The existence of a solution to equation~\eqref{generalmetric_p} follows from the fact that $\Psi_f^L$ is elliptic and selfadjoint. 
The orthogonality of the projection follows similarly as in Sect. \ref{sec:L2} since we have:
\begin{align*}
G^{L}_f\left(X_{\mu}, L\i\left(Tf.\grad^g (p) + p \cdot \Tr^g (S)\right)\right)
=G^{L^2}_f\left(X_{\mu},Tf.\grad^g (p) + p \cdot \Tr^g (S)\right).
\end{align*}

To show well-definedness, we have to show that Eq. \ref{generalmetric_p} has always a solution, and for the difference $q:= p_1 - p_2$ of two solutions $p_1, p_2$ we have $q \in \ker ( L^{-1} (Tf {\rm grad}^g (\cdot) +  \Tr^g (S) \cdot   ))$. As $\Psi_f^L$ is a  elliptic $L^2$-selfadjoint pseudodifferenial operator on a vector bundle $\pi$ over a compact manifold, we know that that $\Gamma(\pi) = \ker (\Psi_f^L) \oplus \im (\Psi_f^L)$, and this decomposition is $L^2$-orthogonal (cf. e.g., Th. III.5.5 in \cite{LawsonMichelsohn} where the statement is made for differential operators instead of pseudodifferential operators. Its proof immediately carries over to all operators satisfying the assumptions in Theorem III.5.2 of that reference, and it is easy to see that $\Psi_f^L$ satisfies them). That means, for $A_f : \operatorname{Vect}_f \rightarrow C^\infty (M)$ defined by $A_f (X) := \div^g (X^T) - g(X , \Tr^g (S) )$, we need to show that $A_f (\operatorname{Vect}_f) \perp  \ker \Psi_f^L$. (Note that $\ker A_f  = T_{f, \mu} \Imm (M,N)$). And indeed, first we have $A_f (\operatorname{Vect}_f)^\perp = \ker (A_f^*)  $ (all adjoints here refer to the $L^2$ metric) and $A_f^* (u) = Tf (\grad^g (u)) - u \Tr^g (S) $. Note that we have $\Psi_f^L = A_f \o L_f^{-1} \o A_f^*$ (which shows again that $\Psi_f^L$ is $L^2$-selfadjoint). If $L_f$ is positive, then it is injective and has a continuous left inverse (which is surjective). If it is moreover elliptic and $L^2$-self-adjoint, then index theory implies that it is bijective and thus an isomorphism. So is $L_f^{-1}$, which is $L^2$-self-adjoint as well.  We want to show $\ker (\Psi_f^L ) = \ker (A_f^*)$. And indeed, as $\im (A_f^*) \perp_{L^2} \ker (A_f)$, $\im(L_f^{-1} A_f^* ) \perp_{W^l} \ker (A_f)$, in particular $\im (L_f^{-1} A_f^* ) \cap \ker (A_f) = \{ 0\} $, thus indeed 
\bea
\label{Kerne}
\ker (\Psi_f^L) = \ker (A_f^*) = \ker (L^{-1} \o A_f).
\eea 
Now let us show the statement above on $q$.
We have $q \in \ker (\Psi_f^L)$, thus Eq. \ref{Kerne} again implies the claim. The continuity of the map can be shown by the usual counting of Sobolev orders taking into account that the right-hand side of Eq.~\eqref{generalmetric_p} is in $H^{k-1}$ (and not in $H^{k-2}$ as the middle term of the chain of equalities in Eq.~\eqref{generalmetric_p} would suggest).

It remains to prove that $P_f^L(X)$ has values in $T_f\Imm_{\mu}(M,N)$. 
Therefore we need to show that
\begin{equation*}
\on{div}^g (P^L_{f}(X)^\top) - \bar g (P^L_{f}(X)^\bot , \Tr^g(S)) = 0.
\end{equation*}
Using the defining equation for $p$ we calculate
\begin{align*}
 &\on{div}^g(P^L_{f}(X)^\top)-  \bar g (P^L_{f}(X)^\bot , \Tr^g(S))\\
 &\qquad=\on{div}^g(X^\top)- \bar g (X^\bot , \Tr^g(S))  - \on{div}^g \left(\left(L^{-1}\left(Tf.\on{grad}^g(p) + p.\Tr^g(S)\right)\right)^\top \right)\\&\qquad\qquad\qquad
 +\bar g\left (L^{-1}\left(Tf.\on{grad}^g(p) + p.\Tr^g(S)\right), \Tr^g(S)\right)
\end{align*}
This yields the differential equation:
\begin{multline*}
 \on{div}^g(X^\top)- \bar g (X^\bot , \Tr^g(S))  = \on{div}^g \left(\left(L^{-1}\left(Tf.\on{grad}^g(p) + p.\Tr^g(S)\right)\right)^\top \right)\\
 -\bar g\left (L^{-1}\left(Tf.\on{grad}^g(p) + p.\Tr^g(S)\right), \Tr^g(S)\right)\;.
\end{multline*}
\end{proof}

Now we would like to show that the projection, extends to a smooth mapping on Sobolev completions of sufficient high order. However, one major building block towards this
result is missing, namely elliptic theory for pseudo differential operators with Sobolev coefficients acting as isomorphism between spaces of Sobolev sections in a certain range of Sobolev orders. For the case of differential operators the relevant results have been proved in \cite{Mueller2015} and used in the proof of Thm.~\ref{thm:wellposedness}.
The proofs for pseudo differential operators will be done in a future paper.
This will lead to the following result:
\begin{conj}\label{assumption}
For each $k>\frac{d}{2}+1$, the operator  $\Psi^L$ depends smoothly on the immersion $f\in\Imm^k(M,N)$ and is invertible as a mapping from
$\Gamma_{H^{k+2-2l}}(f^*TN)$ to $\Gamma_{H^{k}}(f^*TN)$.  
\end{conj}

With this assumption we obtain the following result concerning the smoothness of the projection on the Sobolev completion.

\begin{thm}[{\bf Well-posedness for intermediate metrics}]
 Let $L$ be an elliptic differential operator of order $2l\geq 2$ and $k>\frac{d}{2}+1$.  Assuming that Conj. \ref{assumption} holds, the orthogonal projection $P^L$ extends to a smooth mapping on the Hilbert completions:
 \begin{align*}
  P^L: \Imm^k(M,N)\times T\Imm^k(M,N) &\to T\Imm^k_{\mu}(M,N)\\
  (f,X)&\rightarrow P^L_f(X)\;.
 \end{align*}
\end{thm}

 In \cite{Bauer2011b}, local well-posedness for geodesics in $\on{Imm}(M,N)$ has been shown. The rough procedure is that one first pulls back $\tau_N: TN \rightarrow N$ to $\on{Imm} (M,N) \times M$ by the evaluation map $\on{ev}: \on{Imm} (M,N) \times M \rightarrow N$ given by $\on{ev}(f,m) := f(m)$. On the so obtained bundle $\on{ev}^* \tau_N$ (the bundle whose sections are vector fields along immersions from $M$ to $N$) one considers the pull-back connection $\nabla $ of the Levi-Civita connection on $\tau_N$. The crucial point is that this auxiliary connection is already torsion-free (as it is the pull-back connection of a torsion-free connection). It is not difficult to see that $\nabla$ is the Levi-Civita connection of $\on{Imm} (M,N)$ equipped with the $L^2$ metric. The next step consists then in calculating $\nabla G $ for a higher Sobolev metric $G$, and to express the trilinear form $\nabla G$ by the so-called metric gradients $K$ and $H$ as $\nabla_m G (h,k) = G(K(h,m), k ) = G (m, H(h,k)) $. One has to show that  $H$ and $K$ are continuous bilinear forms. 
 For the Sobolev metric $G^L$ this has been done in \cite{Bauer2011b}). Finally, the geodesic equation is calculated by standard methods as $\nabla_{\partial_t} f_t = \frac{1}{2} H_f (f_t, f_t) - K_f(f_t, f_t)$, and as the connection is torsion-free it follows easily that the Levi-Civita connection $\nabla^{LC}$ of $G_{A,n} $ can be calculated by polarization as $\nabla_X Y (f)= \frac{1}{2} H_f (X, Y) - K_f(X, Y)$, where $H_f$ and $K_f$ are the expressions depending on $n$ and $A$ given by the lemmas in 6.3 and 8.2 of \cite{Bauer2011b}. In our case the geodesic equation for a curve $c$ in $\on{Imm}_\mu (M,N)$ is just 
 $$P_f^L (\nabla^{LC}_t \dot{c} (t)) =0 . $$ 
The previous facts give rise to the following result:
\begin{thm}
 Let $L$ be an elliptic differential operator of order $2l\geq 2$. Under the Conjecture \ref{assumption}, the geodesic spray of the metric $G_L$ on $\on{Imm}^{k+2l}_\mu (M,N)$ is smooth for each $k>\frac{d}{2}+1$, and thus the geodesic equation is locally well-posed on $\on{Imm}^{k+2l}_\mu (M,N)$. The time interval of existence is independent of $k$ and thus this result continues to hold in the smooth category $\on{Imm}_\mu (M,N)$.
 \end{thm}

 \subsection{The theory for strong metrics}

\begin{thm}[{\bf well-posedness for strong metrics}]
Let $l=k>\frac{d}{2}+1$ and consider the metric $G^L$ induced by $L=(1+\Delta)^l$ on $\on{Imm}^k (M,N)$. Then the induced metric on $ \on{Imm}_\mu^k (M,N) $ is a strong metric again, and its geodesic equation for $G$ is locally well-posed. The time interval of existence is independent of $k$ and thus this result continues to hold in the smooth category $\on{Imm}_\mu (M,N)$.
 \end{thm}

 \begin{proof}
 We have seen above that $G$ is a strong metric on $\on{Imm}^k (M,N)$. The statement follows from the following well-known fact: If $X$ a Hilbert manifold and $Y$ a Hilbert submanifold of $X$ modelled on closed linear subspace, then the induced metric on $Y$ is strong again. This is because if $A$ is the operator appearing in the definition of the strong metric on $X$ and if we choose the orthogonal projection $P$ on $TY$ then $\tilde{A}:= P \o A $ is the searched-for intermediating operator appearing in the definition of 'strong metric' on $Y$. Knowing that the metric $G$ restricted to $ \on{Imm}_\mu^k (M,N) $ is strong, we can invoke \cite{Lang1999}, VIII.4.2 and VIII.5.1, to show that the geodesic equation is locally well-posed.
\end{proof}

\section{Examples}\label{sec:ex}

\subsection{The group of volume preserving diffeomorphisms}\label{sec:ex:SDiff}
In the following we want to consider the special case  that $M$ equals $N$. Then the space of all volume preserving embeddings equals the group of volume preserving diffeomorphisms; i.e., ${\rm Emb}_\mu (M,N) = {\rm Diff}_\mu(M,N)$. 
The geodesic equation for the $L^2$-metric then simplifies to Euler's equations for the motion of an incompressible fluid, see \cite{Arnold1966}. 
Local well-posedness for this equation has been shown by Ebin and Marsden in \cite{EM1970}. For strong metrics; i.e., for $l=k$, even global well-posedness is true:
\begin{cor}[cf. \cite{mBfV}, Remark after Cor. 7.6]\label{thm:SDiff}
Let $k>\frac{d}{2}+1$. The space $\Diff_\mu^k(M)$ equipped with the right-invariant Sobolev metric of order $k$ is a geodesically and metrically complete space. 
\end{cor}

\begin{proof}
It has been established recently \cite{mBfV} that $ (\Diff^k(M,N), G^k) $ is geodesically and metric complete, see also \cite{MM2013,BEK2015,EK2014}.  
As $k> d/2 +1$, the map $f \mapsto f^*\vol$ is continuous, thus $\Diff_\mu(M)$ is a closed subset of the metrically complete space $\Diff(M)$, see also  \cite{EM1970}. Hence it is metrically complete and thus geodesically complete by \cite{Lang1999}, Prop. VIII.6.5. 
\end{proof}

\subsection{The space of constant speed curves}\label{sec:ex:curves}
In this part we want to consider the special case $M=S^1$, $N=\mathbb R^2$ and $\mu=\frac{2\pi}{\ell_c}d\theta$ and we will show that we regain the formulas of \cite{Preston2011,Preston2012,PS2013}.
We start with the $L^2$-metric and we want to consider the geodesic equation from Sec.~\ref{sec:L2} in this much simpler situation:
\begin{cor}
On the space $\Imm_{\mu}(S^1,\mathbb R^2)$ the geodesic equation of the $L^2$-metric reads as
$$\left\{\begin{aligned}
c_{tt} &= p'.c' + p.c''= (pc')',\\
  p''-\|c''\|^2 p &=  -\|c'_t\|^2  \,.
\end{aligned}\right.$$ 
\end{cor}
\begin{rem*}
 Note, that this equation is equal to the equation studied in \cite{Preston2011,Preston2012}. The main difference to the general situation is the constant sign of the right hand side $-\|c'_t\|^2$. In 
\cite{Preston2011,Preston2012} this was used to show local wellposedness of this equation.
\end{rem*}
\begin{proof}
This follows directly from the formula of the geodesic equation in Sec.~\ref{sec:L2}, using the fact that the metric $g=\frac{1}{|c'|^2}$ is constant on $\Imm_{\mu}(S^1,\mathbb R^2)$.
Thus we have
\begin{equation}
\on{Tr}\left(\partial_t\left(g(t)^{-1}\right)\langle \nabla^{g} f_t,Tf   \rangle\right)=\on{Tr}\left(g^{-1}\langle \partial_t\left( \nabla^{g(t)}\right) f_t,Tf   \rangle\right)=0\qedhere
\end{equation}
\end{proof}

The observation that the not only the volume form, but also the metric is constant on $\Imm_{\mu}(S^1,\mathbb R^2)$ continues to have a large influence also for the higher order metrics.
We now want to study the operator $\Psi$ that is used to define the orthogonal projection. To simplify the notation we assume $|c'|=1$. Then we have:
\begin{align*}
 \Psi^L_c(p)=\partial_{\theta} \left(\left((1-\partial^2_{\theta})^{-l}\left(p'.c'  + p.c''\right)\right)^\top \right)-\bar g\left ((1-\partial^2_{\theta})^{-l}\left(p'.c'  + p.c''\right),c''\right)\;.
\end{align*}
We can now further rewrite this to obtain:
\begin{align*}
 \Psi^L_c(p)&=\partial_{\theta} \left(\left((1-\partial^2_{\theta})^{-l}\left(\partial_{\theta}(p.c') \right)\right)^\top \right)-\bar g\left ((1-\partial^2_{\theta})^{-l}\left(\partial_{\theta}(p.c') \right),c''\right)
 \\&=\partial_{\theta}\bar g \left((1-\partial^2_{\theta})^{-l}\left(\partial_{\theta}(p.c') \right),c' \right)-\bar g\left ((1-\partial^2_{\theta})^{-l}\left(\partial_{\theta}(p.c') \right),c''\right)\\
 &=\bar g \left(\partial_{\theta}(1-\partial^2_{\theta})^{-l}\left(\partial_{\theta}(p.c') \right),c' \right)=\bar g \left((1-\partial^2_{\theta})^{-l}\partial^2_{\theta}(p.c') ,c' \right)
\end{align*}
Note that for $l=0$ this gives $\Psi^0_c(p)=p''-|c''|^2p$. In the $C^1$-topology the existence of solutions to this equations has been shown in \cite{PS2013}.

Recently it was shown in \cite{Bruveris2014}, that the geodesic equation on the space of curves is globally well-posed for $l\geq 2$. 
Using this result, one would expect to obtain the analogue of Cor.~\ref{thm:SDiff} for the space of constant speed curves.

\appendix
\section{Variational formulas.}
In this appendix we will collect some variational formulas that we used throughout the article. For proofs of these results using a similar notation we refer to \cite{Bauer2011b}.

\begin{lem}\nmb.{1.5}	 {\rm \cite[Lemma 5.5 and Lemma 5.6]{Bauer2011b}}
The differential of the pullback metric
\begin{equation*}\left\{ \begin{array}{ccl}
\Imm &\to &\Gamma(S^2_{>0} T^*M),\\
f &\mapsto &g=f^*\bar g
\end{array}\right.\end{equation*}
is given by
\begin{align}
D_{(f,h)} g&= 2\on{Sym}\bar g(\nabla h,Tf) = -2 \bar g(h^\bot,S)+2 \on{Sym} \nabla (h^\top)^\flat 
\\& = -2 \bar g(h^\bot,S)+ \L_{h^\top} g.
\end{align}
The differential of the inverse of the pullback metric
\begin{equation*}\left\{ \begin{array}{ccl}
\Imm &\to &\Ga\big(L(T^*M,TM)\big),\\
f &\mapsto &g\i=(f^*\bar g)\i
\end{array}\right.\end{equation*}
is given by
\begin{align}
D_{(f,h)} g\i = D_{(f,h)} (f^*\bar g)\i =2 \bar g(h^\bot, g\i S  g\i) + \mathcal L_{h^\top}(g\i)
\end{align}
\end{lem}

\begin{lem}\nmb.{1.6}{\rm \cite[Lemma 5.7]{Bauer2011b}}
\label{partialtvolume}
The differential of the volume density
\begin{equation*}
\left\{ \begin{array}{ccl}
\Imm &\to &\Vol(M),\\
f &\mapsto &\vol(g)=\vol(f^*\bar g)
\end{array}\right.\end{equation*}
is given by
\begin{equation*}
D_{(f,h)} \vol(g) = 
\Tr^g\big(\bar g(\nabla h,Tf)\big) \vol(g)=
\Big(\on{div}^{g}(h^{\top})-\bar g\big(h^{\bot},\Tr^g(S)\big)\Big) \vol(g).
\end{equation*}
\end{lem}

Here, the last equation is easy to see decomposing $h$ in its tangential and normal part.

\begin{lem}\nmb.{1.7}{\rm \cite[3.11 and Lemma 5.9]{Bauer2011b}}
The Bochner-Laplacian is defined by
$$\Delta B = \nabla^*\nabla B = - \on{Tr}^g(\nabla^2 B).$$
for any tensor fields $B$.
It is a smooth section $f\mapsto \De^{f^*\bar g}$ of the bundle 
$$
L(T\Imm(M,N);T\Imm(M,N))\to\Imm(M,N).
$$
Its derivative can be expressed by  
the covariant derivative explained in section:
For $\De \in \Ga\big(L(T\Imm;T\Imm)\big)$, $f \in \Imm$ and $f_t,h \in T_f\Imm$ one has
\begin{align}
(\nabla_{f_t} \Delta)(h) &=
\on{Tr}\big(g\i.(D_{(f,f_t)}g).g\i \nabla^2 h\big) 
-\nabla_{\big(\nabla^*(D_{(f,f_t)} g)+\frac12 d\on{Tr}^g(D_{(f,f_t)}g)\big)^\sharp}h \\&\qquad
+\nabla^*\big(R^{\bar g}(f_t,Tf)h\big)
-\Tr^g\Big( R^{\bar g}(f_t,Tf)\nabla h \Big).
\end{align}
\end{lem}
The Bochner-Laplacian coincides with the de~Rham-Laplacian on the space of functions.

\section{More on covariant derivatives}\label{variational}
Let $\nabla$ denote any kind of induced covariant derivative which comes from the Levi-Civita 
derivative of $\bar g$. Since we will use an induced covariant derivative for several kinds of 
tensor bundles on $\Imm(M,N)$, let us explain the setup of \cite[3.7 and 4.2]{Bauer2011b} a bit, 
which uses the detailed setup 
of \cite[sections 19.12, 22.9]{MichorG}. If we want to be specific we will write
$\nabla^g, \nabla^{\bar g}$ for the \emph{Levi-Civita covariant derivatives} on $(M,g)$
and $(N,\bar g)$, respectively. 
For any manifold $Q$ and vector field $X$ on $Q$, one has
\begin{align}
\nabla^g_X:C^\infty(Q,TM) &\to C^\infty(Q,TM), & h &\mapsto \nabla^g_X h \\
\nabla^{\bar g}_X: C^\infty(Q,TN) &\to C^\infty(Q,TN), & h &\mapsto \nabla^{\bar g}_X h.
\end{align}
From the properties listed in \cite[section 3.7]{Bauer2011b} we just repeat the following:
\begin{enumerate}
\item[(1)]
$\pi \o \nabla_X h = \pi \o h$, where $\pi$ is the projection 
of the tangent space onto the base manifold. 
\item[(5)] 
For any manifold $\widetilde Q$ and smooth mapping 
$q:\widetilde Q \to Q$ and $Y_y \in T_y \widetilde Q$ one has
$\nabla_{Tq.Y_y}h=\nabla_{Y_y}(h \o q)$. If $Y \in \X(Q_1)$ and $X \in \X(Q)$ are $q$-related, then 
$\nabla_Y(h \o q) = (\nabla_X h) \o q$.
\end{enumerate}
The two covariant derivatives $\nabla^g_X$ and $\nabla^{\bar g}_X$ 
combine to yield a covariant derivative $\nabla_X$ acting on
$C^\infty(Q,T^r_sM \otimes TN)$ in the usual way.

The covariant derivative $\nabla^{\bar g}$ induces a
\emph{covariant derivative over immersions} as follows.
Let $Q$ be a smooth manifold. Then one identifies
\begin{align}
&h \in  C^\infty\big(Q,T\Imm(M,N)\big) && \text{and} && X \in \X(Q)
\intertext{with}
&h^{\wedge} \in C^\infty(Q \x M, TN) && \text{and} && (X,0_M) \in \X(Q \x M).
\end{align}
As above one has the covariant derivative
$$\nabla^{\bar g}_{(X,0_M)} h^{\wedge} \in C^\infty\big(Q \x M, TN).$$
Thus one can define
$$\nabla_X h = \left(\nabla^{\bar g}_{(X,0_M)} h^{\wedge}\right)^{\vee} \in C^\infty\big(Q,T\Imm(M,N)\big).$$
This covariant derivative is torsion-free; see \cite[section~22.10]{MichorG}.
It respects $\bar g$ and $\bar G$ but in general does not respect any of the invariant metrics $G$ 
used above.  
The special case $Q=\R$ will be important to formulate the geodesic equation. 
The expression that will be of interest in the formulation of the 
geodesic equation is $\nabla_{\p_t} f_t$, which is 
well-defined when $f:\R \to \Imm(M,N)$ is a path of immersions and $f_t: \R \to T\Imm(M,N)$ is its velocity. 
Another case of interest is $Q = \Imm(M,N)$. Let $h, k, m \in \X(\Imm(M,N))$. Then the covariant 
derivative $\nabla_m h$ is well-defined and tensorial in $m$.
Requiring $\nabla_m$ to respect the grading of the spaces of multilinear maps, to act as a derivation 
on products and to commute with compositions of multilinear maps, one obtains 
as above a covariant 
derivative $\nabla_m$ acting on all mappings into 
the natural bundles of multilinear mappings over $\Imm(M,N)$.
We shall use it as background (static) covariant derivative.
In particular, $\nabla_m L$ and $\nabla_m G$ are well-defined for 
\begin{align}
L \in \Ga\big(L(T\Imm(M,N);T\Imm(M,N))\big), \quad
G \in \Ga\big(L^2_{\on{sym}}(T\Imm(M,N);\R)\big)
\end{align}
by the usual formulas
\begin{align}
(\nabla_m P)(h) &= \nabla_m\big(P(h)\big) - P(\nabla_mh), \\
(\nabla_m G)(h,k) &= \nabla_m\big(G(h,k)) - G(\nabla_m h,k) - G(h,\nabla_m k). 
\end{align}


\end{document}